\documentclass[12pt]{article}
\usepackage{amsmath,amssymb,amsfonts,amsthm,amscd,graphicx,psfrag,epsfig}
\newtheorem{theorem}{Theorem}[section]

\newtheorem{theorem-definition}[theorem]{Theorem-Definition}
\newtheorem{theorem-construction}[theorem]{Theorem-Construction}
\newtheorem{lemma-definition}[theorem]{Lemma--Definition}
\newtheorem{lemma-construction}[theorem]{Lemma--Construction}
\newtheorem{lemma}[theorem]{Lemma}
\newtheorem{proposition}[theorem]{Proposition}
\newtheorem{corollary}[theorem]{Corollary}
\newtheorem{conjecture}[theorem]{Conjecture}
\newtheorem{definition}[theorem]{Definition}

\newcommand{\old}[1]{}

\newcommand{\Z}{{\mathbb Z}}
\newcommand{\R}{{\mathbb R}}

\newcommand{\C}{{\mathbb C}}
\newcommand{\T}{{\mathbb T}}

\newcommand{\ra}{\rightarrow}
\newcommand{\be}{\begin{equation}}
\newcommand{\ee}{\end{equation}}
\newcommand{\bt}{\begin{theorem}}

\newcommand{\et}{\end{theorem}}
\newcommand{\bd}{\begin{definition}}
\newcommand{\ed}{\end{definition}}
\newcommand{\bp}{\begin{proposition}}
\newcommand{\ep}{\end{proposition}}

\newcommand{\bl}{\begin{lemma}}
\newcommand{\el}{\end{lemma}}
\newcommand{\bc}{\begin{corollary}}
\newcommand{\ec}{\end{corollary}}
\newcommand{\bcon}{\begin{conjecture}}
\newcommand{\econ}{\end{conjecture}}

\usepackage{graphicx}
\usepackage{amsrefs}
\usepackage{geometry}
\usepackage{graphicx}
\usepackage{bm}
\usepackage{amsmath,amsthm,amssymb,graphicx,mathtools,tikz-cd,hyperref}
\usepackage[]{subcaption}
\usetikzlibrary{positioning}
\usepackage{amssymb}

\begin{document}

%%
%% The title of the paper goes here.  Edit to your title.
%%

\title{Spectra of biperiodic planar networks}
\author{Terrence George } 
\newcommand{\Addresses}{{% additional braces for segregating \footnotesize
  \bigskip
  \footnotesize

  Terrence George, \textsc{Department of Mathematics, Brown University,
    Providence, Rhode Island 02912}\par\nopagebreak
  \textit{E-mail address}: \texttt{gterrence@math.brown.edu}
}}

\maketitle
\bibliographystyle{amsxport}
\begin{abstract}
A biperiodic planar network is a pair $(G,c)$ where $G$ is a graph embedded on the torus and $c$ is a function from the edges of $G$ to non-zero complex numbers. Associated to the discrete Laplacian on a biperiodic planar network is its spectrum: a triple $(C,S,\nu)$,  where $C$ is a curve and $S$ is a divisor on it. We give a complete classification of networks (modulo a natural equivalence) in terms of their spectral data. The space of networks has a large group of cluster automorphisms arising from the $Y-\Delta$ transformations. We show that the spectrum provides action-angle coordinates for the discrete cluster integrable systems defined by these automorphisms.
\end{abstract}
\tableofcontents
\section{Introduction}
A planar resistor network is a pair $(\tilde{G},\tilde{c})$ where $\tilde{G}$ is a planar graph and $\tilde{c}$ is a conductance function that assigns a non-zero complex number to each edge of $\tilde G$, defined up to multiplication by a global constant. It is said to be biperiodic if translations by $\mathbb Z^2$ act on $(\tilde{G},\tilde{c})$ by isomorphisms. This  is equivalent to the data of the quotient $(G,c):=(\tilde{G},\tilde{c})/\mathbb Z^2$, where $G$ is a graph on a torus. Hereafter we assume that our networks are on a torus.\\

The fundamental operator in the study of networks is the discrete Laplacian. It has a certain spectrum, defined below, and the main goal of this paper is to show that this is a birational isomorphism with a certain moduli space of curves and divisors and therefore provides a way to classfiy networks. While in typical geometric or probabilistic applications the conductances are always positive real numbers, the algebraic nature of the problem leads us to consider general (nonzero) complex conductances.\\

There is a natural equivalence relation on networks, defined by certain local rearrangements of the graph and its conductances,
which preserves the spectrum. To define this equivalence relation, let us start by defining a zig-zag path. A zig-zag path on $G$ is a path that alternately turns maximally left or right. A resistor network $G$ is minimal (\cite{CdV94}, \cite{CIM98}) if lifts of any two zig-zag paths to $\tilde{G}$ do not intersect more than once and any lift of a zig-zag path has no self intersections. Minimality is a mild assumption on networks since any network may be reduced to a minimal one by certain elementary moves without affecting its electrical properties. The Newton polygon of a minimal resistor network is the unique integral polygon whose primitive edges are given by the homology classes of zig-zag paths in cyclic order. Since zig-zag paths come in pairs related by flipping the orientation, the Newton polygon of a network is always centrally symmetric.\\

We say that two minimal networks $(G_1,c_1)$ and $(G_2,c_2)$ are topologically equivalent if there is a sequence of $Y-\Delta$ moves that takes the underlying graph $G_1$ to the graph $G_2$. Topological equivalence classes of networks are parameterized by centrally symmetric Newton polygons (\cite{GK12}). In particular, any two minimal resistor networks with the same Newton polygon are related by a sequence of elementary transformations called $Y-\Delta$ transformations. \\

Two networks $(G_1,c_1)$ and $(G_2,c_2)$ are electrically equivalent if there is a sequence of $Y-\Delta$ moves that takes the network $(G_1,c_1)$ to the network $(G_2,c_2)$. Goncharov and Kenyon (\cite{GK12}) constructed the resistor network cluster variety $\mathcal R^0_N$ that parameterizes electrical equivalence classes of resistor networks that lie in the same topological equivalence class associated to the polygon $N$ as follows: 
A centrally symmetric integral polygon $N$ determines a finite collection of minimal resistor networks whose Newton polygon is $N$, related by $Y-\Delta$ transformations. To each minimal resistor network $G$ is associated a complex torus $(\mathbb C^*)^{\text{number of edges of }G-1}$, which parameterizes conductance functions on $G$. The $Y-\Delta$ move $G_1 \ra G_2$ induces a birational map between the complex tori associated to $G_1$ and $G_2$. $\mathcal{R}^0_N$ is obtained by gluing the complex tori using these birational maps. \\

Goncharov and Kenyon further showed that $\mathcal{R}^0_N$ is a Lagrangian subvariety of an algebraic completely integrable Hamiltonian system $\mathcal X_N^0$ associated to the dimer model. Let $S_N$ be the moduli space of triples $(C,S,\nu)$ where $C$ is the vanishing locus of a Laurent polynomial $P(z,w)$ with Newton polygon $N$, $S$ is a degree $g$ effective divisor on $C$ (where $g=$ number of interior lattice points in $N$) and $\nu$ is a parameterization of the points at infinity of $C$. Goncharov and Kenyon constructed the spectral map $\mathcal X_N^0 \ra \mathcal S_N$ and showed that it is a birational isomorphism. Fock (\cite{Fock15}) constructed an explicit inverse map in terms of theta functions on the Jacobian of $C$. In this construction, the elementary transformation in dimer model (the spider move) is described by Fay's trisecant identity.\\

Associated to the Laplacian on a biperiodic planar network is its spectrum $\mathcal R^0_N \ra \mathcal{S}_N$, where $\mathcal S_N$ is defined as in the previous paragraph, but with the divisor $S$ now of degree $g=$ number of interior lattice points in $N$-1. Let $\mathcal S_N'$ be the subspace where $P(z,w)$ satisfies
\begin{enumerate}
\item $P(1,1)=0$ and the point $(1,1)$ is a node;
\item $\sigma:(z,w) \mapsto (\frac{1}{z},\frac{1}{w})$ is an involution on $C$,
\end{enumerate}
and the divisor $S$ satisfies
$$
S+\sigma(S)-q_1-q_2 \equiv K_{\hat{C}},
$$
where $\hat{C}$ is the normalization of $C$, $q_1,q_2$ are the points in the fiber of the node at $(1,1)$ and $K_{\hat{C}}$ is the canonical divisor class on $\hat{C}$.
Our main result is the following complete classification of biperiodic planar resistor networks in terms of their spectral data:
\begin{theorem}\label{thm1}
The spectral map is a birational isomorphism $\mathcal{R}_N^0 \ra \mathcal S_N'$. 
\end{theorem}

Along the way, we provide an explicit description of oriented cycle rooted spanning forests of $G$ (OCRSFs) whose homology classes are boundary lattice points of $N$ (Lemmas \ref{crsfextremal}, \ref{crsfexternal}), analogous to results for dimers in \cite{Bro12}, \cite{GK12}. In particular, we show that every OCRSF corresponding to a boundary lattice point is a union of cycles (Corollary \ref{cyc}).\\

We construct an explicit inverse spectral map (see (\ref{invmap})). A key player in the construction is the theta function on the Prym variety of $\hat{C}$. The $Y-\Delta$ transformation is described by Fay's quadrisecant identity (\cite{Fay89}). Further we show that the inverse map is compatible with $Y-\Delta$ transformations (Theorem \ref{ydcomp}).\\

\begin{figure}\label{amoeba}
\centering
\includegraphics[width=0.5\textwidth]{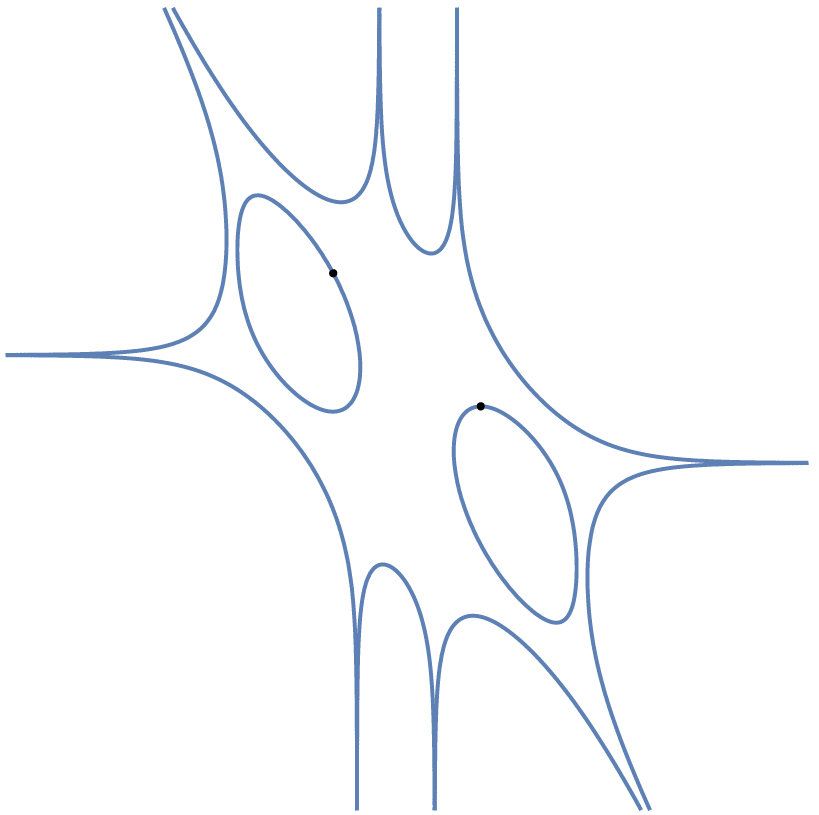}
\caption{The divisor $S$ on the amoeba of the spectral curve.}\label{amoeba}
\end{figure}
Since the $Y-\Delta$ move involves subtraction free rational expressions, the set of positive real valued points of the cluster variety is well defined, which we denote by $\mathcal R_N(\R_{\geq 0})$. This subspace is important for probabilistic applications. For a positive real valued point, the spectrum $(C,S,\nu)$ has the following additional properties (see \cite{K17}):
\begin{enumerate}
    \item $C$ is a simple Harnack curve (\cite{M}). Compact ovals (connected components) of $C$ are in bijection with interior lattice points of $N$.
    \item The oval corresponding to the origin is degenerated to a real node.
    \item $S$ has a point in each of the other compact ovals.
\end{enumerate}
The spectral curves of genus zero correspond to the isoradial networks studied in \cite{K02}. In this case, the inverse spectral map recovers Kenyon's results expressing the conductances in terms of tangents, and the quadrisecant identity reduces to the triple tangent identity. For a different generalization of isoradial networks to the case of the massive Laplacian on isoradial graphs, see \cite{BdeTR17}.\\

Consider the map $C(\C) \ra \mathcal A(C), (z,w) \mapsto (\log |z|,\log |w|) \subset \R^2$ from the $\C-$valued points of $C$ to its amoeba $\mathcal A(C)$. For a simple Harnack curve, this map is a homeomorphism from the compact ovals to the boundaries of the holes of the amoeba, and therefore provides a way to depict the divisor $S$ (see Figure \ref{amoeba} for an example, where the network is a $2 \times 1$ fundamental domain of the triangular lattice).\\

A sequence of $Y-\Delta$ moves that takes a graph $G$ to itself gives rise to a birational automorphism (called a cluster modular transformation) of $\mathcal R_N$, where $N$ is the Newton polygon of $G$. A cluster modular transformation provides a discrete integrable system on $\mathcal R_N$. For example, if we consider the honeycomb lattice, and do the $Y-\Delta$ move at the downward triangles, we obtain the cube recurrence studied by Carroll and Speyer (\cite{CS04}, see also \cite{GK12} section 6.3). We show that cluster modular transformations are linearized on the Prym variety of $C$ (Theorem \ref{lin}). In the case of positive real conductances, we may view this as moving each point along the boundary of the corresponding hole in the amoeba.\\

\textbf{Acknowledgements.} We thank Giovanni Inchiostro, Rick Kenyon and Xufan Zhang for helpful discussions.
\section{The dimer model}
\subsection{Line bundles with connection}

A \textit{surface graph} $\Gamma$ on a torus $\T$ is a graph embedded on $\T$ such that each face is contractible. A \textit{line bundle with connection} $(V,\phi)$ on $\Gamma$ is the data of a complex line $V_v \cong \C$ at each vertex of $\Gamma$ along with isomorphisms called \textit{parallel transport} $\phi_{v v'}:V_v \ra V_{v'}$ for each edge $\langle v,v' \rangle$ such that $\phi_{v' v}=\phi_{v v'}^{-1}$. Two line bundles with connection $(V,\phi)$ and $(V',\phi')$ are \textit{gauge equivalent} if there exists isomorphisms $\psi_v:V_v \ra V'_v$ such that for all edges, the following diagram commutes.
\[
\begin{tikzcd}
V_v \arrow[r, "\phi_{vv'}"] \arrow[d,"\psi_v"]
& V_{v'} \arrow[d, "\psi_{v'}" ] \\
V'_{v} \arrow[r, "\phi'_{vv'}" ]
& V'_{v'}
\end{tikzcd}
\]
If $L$ is an oriented loop in $\Gamma$, the \textit{monodromy} $m(L)$ of $(V,\phi)$ around $L$ is the composition of the parallel transports around $L$. A line bundle with connection is \textit{flat} if the monodromy around the boundary of any face of $\Gamma$ is trivial.\\
The moduli space of line bundles with connection on $\Gamma$ modulo isomorphisms is denoted $\mathcal{L}_\Gamma$. Let $\mathcal L_\Gamma^{\text{flat}}$ be the subspace of flat connections. The monodromies around loops in $\Gamma$ give rise to isomorphisms such that the following diagram commutes:
\[
\begin{tikzcd}
\mathcal L_\Gamma^{\text{flat}}  \arrow[hookrightarrow,r] \arrow[d,"\cong"]
& \mathcal{L}_\Gamma \arrow[d, "\cong" ] \\
H^1(\T,\C^*) \arrow[hookrightarrow,r]
& H^1(\Gamma,\C^*)
\end{tikzcd}
\]

A \textit{dimer cover} (or \textit{perfect matching}) of $\Gamma$ is a collection of edges of $\Gamma$ such that every vertex is adjacent to a unique edge. A dimer cover $M$ on $\Gamma$ gives a $1$-chain $\omega_M$ on $\Gamma$. If $M_0$ is another dimer cover, $\omega_M-\omega_{M_0}$ is a cycle and therefore determines a homology class in $H_1(\Gamma,\Z)$. Under the projection $H_1(\Gamma,\Z) \ra H_1(\T,\Z)$, we obtain a homology class $[M]\in H_1(\T,\Z).$ The Newton polygon of the dimer model is 
$$
N:=\text{Conv }\{[M] \in H_1(\T,\Z): M \text{ is a dimer cover}\}.
$$
$N$ depends on the choice of reference dimer cover $M_0$. Changing the reference matching corresponds to translating the polygon $N$.
$M \mapsto [M]$ gives a well defined map from the set of dimer covers to the integer lattice points in $N$. 

\subsection{Zig-zag paths on bipartite graphs and minimality}

A \textit{zig-zag path} on a bipartite torus graph $\Gamma$ is a path that turns maximally right at black vertices and maximally left at white vertices. Let us denote by $\mathcal Z_\Gamma$ the set of all zig-zag paths in $\Gamma$. We say that $\Gamma$ is \textit{minimal} if in the universal cover $\tilde{\Gamma}$, zig-zag paths have no self intersections and no pairs of zig-zag paths oriented in the same direction meet twice. \\
Suppose $\Gamma$ is a minimal bipartite graph on a torus. Each path $\alpha \in Z_{\Gamma}$ gives us a homology class $[\alpha] \in H_1(\T,\Z)$ which is an integral pimitive vector on a side of the Newton polygon $N$. The zig-zag paths taken in cyclic order correspond to cyclically ordered primitive integral vectors in  the boundary of the Newton polygon. Therefore an edge of $N$ corresponds to a family of zig-zag paths, each with homology class equal to the primitive integral edge vector of the edge.

\subsection{External dimer covers}
In this section, we collect some results about dimer covers from \cite{Bro12},\cite{GK12}. Let $\Gamma$ be a minimal bipartite graph on a torus.
We say that a dimer cover $M$ is \textit{extremal} if $[M]$ is a vertex of the Newton polygon.
If $b$ is any black vertex in $\Gamma$, we define the \textit{local zig-zag fan} $\Sigma_b$ at $b$ to be the complete fan of strongly convex rational polyhedral cones in $H_1(\T,\Z)$ whose rays are generated by  homology classes of those zig-zag paths in $\Gamma$ that contain $b$. 
\begin{lemma}[\cite{Bro12}]\label{span}
The rays corresponding to two families of zig-zag paths span a two dimensional cone $\Sigma_v$ if and only if there is an edge incident to $v$ at which the two families intersect.
\end{lemma}

The \textit{global zig-zag fan} of $\Gamma$ is the fan whose rays are generated by the homology classes of all zig-zag paths on $\Gamma$. The identity map in $H_1(\T,\Z)$ defines a map of fans $i_b:\Sigma \ra \Sigma_b$. If $\sigma$ is any two dimensional cone in $\Sigma$, $i_b(\sigma)$ is contained in a unique two dimensional cone in $\Sigma_b$ which we call $\sigma_b$. $\sigma_b$ corresponds to a unique edge incident to $b$, given by the intersection of the two zig-zag paths through $b$ whose rays in $\Sigma_b$ form the boundary of $\sigma_b$. Define the 1-chain $\omega(\sigma_b)$ to be $1$ on the edge $\langle w,b \rangle$ and $0$ on all other edges. Define 
$$
\omega(\sigma)=\sum_{b \in V(\Gamma) \text{ black }}\omega(\sigma_b). 
$$
Two dimensional cones in $\Sigma$ are in bijection with vertices of the Newton polygon: If $\sigma$ is a two dimensional cone in $\Sigma$, let $E_1$ and $E_2$ be the edges of $N$ whose associated rays form the boundary of $\sigma$ in $\Sigma$. Then $E_1$ and $E_2$ occur in cyclic order and therefore there is a vertex $V$ between them in $N$.

\begin{theorem}[\cite{Bro12},\cite{GK12}]\label{extremaldimer}
$\omega_V:=\omega(\sigma)$ is the unique extremal dimer cover associated to the vertex $V$ of $N$ that corresponds to $\sigma$.
\end{theorem}

We say that a dimer cover $M$ is \textit{external} if $[M]$ is a boundary lattice point of $N$. To a zig-zag path $\alpha$ we associate a 1-form $\omega_\alpha$ that is $1$ on edges $e$ in $\alpha$ that are oriented the same way as $\alpha$ and $0$ on edges not in $\alpha$. If $M$ is external, $[M]$ lies on an edge $E$ of $N$, which corresponds to a family of zig-zag paths $\{\alpha_k\}$. Let $E=\langle V_1,V_2\rangle$, where $V_1,V_2$ are vertices of $N$ such that $V_2$ is the vertex after $V_1$ when the boundary of $N$ is traversed counterclockwise. 
\begin{theorem}[\cite{Bro12},\cite{GK12}]\label{externaldimer}
Let $A$ be a subset of the family of zig-zag paths $\{\alpha_k\}$ corresponding to $E$. The external dimer covers on $E$ are of the form 
$$
\omega_A:=\omega_{V_1}+\sum_{\alpha_k \in A}\omega_{\alpha_k}.
$$
In particular, $\omega_{V_2}=\omega_{V_1}+\sum_{k}\omega_{\alpha_k}$, and the number of dimer covers corresponding to a boundary lattice point of $N$ is a binomial coefficient.
\end{theorem}

\section{Resistor networks}
A \textit{resistor network} is a pair $(G,c)$ where $G$ is a surface graph on $\T$ and $c:E(G)\ra \C^*$ is a function defined modulo global multiplication by a non-zero scalar.

Associated to $G$ is a bipartite graph $\Gamma_G$ obtained by superposing $G$ and its dual graph $G^\vee$. The vertices and faces of $G$ become the black vertices of $\Gamma_G$ and the edges of $G$ become the white vertices of $\Gamma_G$. Applying Euler's formula on $\T$ to $G$ we see that $\Gamma_G$ has equal number of white and black vertices.
\subsection{The resistor network cluster variety}
A conductance function $c$ determines a line bundle with connection $V(c)$ on $\Gamma_G$ as follows:\\
The weight assigned to an edge of $\Gamma_G$ incident to a vertex of $G$ is the conductance of that edge in $G$. The weight of an edge incident to a face of $G$ is $1$. Composing the isomorphisms $\mathcal L^{\text{flat}}_G \cong H^1(\T,\C^*) \cong \mathcal L^{\text{flat}}_{\Gamma_G}$, we see that the moduli spaces of line bundles with flat connections on $G$ and $\Gamma_G$ are canonically isomorphic. $\mathcal{R}_G \subset \mathcal L_G$ of line bundles of the form $V(c)\otimes i$, where $c$ is a conductance function on $G$ and $i\in \mathcal L^{\text{flat}}_G$ is a closed subvariety.

A $Y-\Delta$ transformation(\cite{Kenn1899})$G_1 \ra G_2$ is given by replacing a $Y$ in the graph $G_1$ with a triangle as shown in Figure \ref{et}. Any two minimal resistor networks with Newton polygon $N$ are related by $Y-\Delta$ moves. A $Y-\Delta$ move $G_1 \ra G_2$ induces a birational map $\mathcal{L}_{\Gamma_{G_1}} \ra \mathcal{L}_{\Gamma_{G_2}}$.  Gluing the $\mathcal{L}_{\Gamma_{G}}$ using these birational maps, we obtain a cluster Poisson variety $\mathcal{X}_N$.

The birational map $\mathcal{L}_{\Gamma_{G_1}} \ra \mathcal{L}_{\Gamma_{G_2}}$ restricted to $\mathcal{R}_{G_1} \ra \mathcal{R}_{G_2}$ is given in the notation of Figure \ref{et} by
\begin{align*}
A=\frac{bc}{a+b+c},B=\frac{ac}{a+b+c},C=\frac{ab}{a+b+c}.
\end{align*}
Gluing the subvarieties $\mathcal{R}_{G}$ using these birational isomorphisms, we obtain a cluster subvariety $\mathcal{R}_N$ of $\mathcal{X}_N$. Quotienting by the moduli space of the flat connections, we get $\mathcal{R}_N^0 \subset \mathcal{X}_N^0$, called the \textit{resistor network cluster variety}. 

\subsection{The line bundle Laplacian}
Let $(G,c)$ be a resistor network and let $i \in \mathcal L^{\text{flat}}_G$. The line bundle Laplacian is the linear operator 
$\Delta=\Delta(c,i):\mathbb C^{V(G)} \ra \mathbb C^{V(G)}$ defined by 
$$
\Delta(f)(v):=\sum_{v' \sim v }c(v,v')(f(v)-i_{v'v}f(v')).
$$
An \textit{oriented cycle rooted spanning forest} (OCRSF) $\gamma$ of $G$ is a collection of edges of $G$ such that each connected component of $\gamma$ has the same number of vertices and edges (so that each connected component has a unique cycle), along with a choice of orientation for each cycle in $\gamma$. Since two distinct cycles in $\gamma$ cannot intersect, if $\eta$ is a cycle in $\gamma$, every cycle has homology class $\pm [\eta]$. The weight of an $OCRSF$ $\gamma$ is defined to be $wt(\gamma)=\prod_{e \in \gamma}c(e)$.  
\begin{theorem}[\cite{K10}]\label{pfnlap}
$$\text{det }\Delta = \sum_{ \text{OCRSFs } \gamma}wt(\gamma)\prod_{\text{Cycles }\eta \in \gamma}(1-m(\eta)),$$
where $m(\eta)$ is the monodromy of $i$ along the cycle $\eta$.
\end{theorem}

An OCRSF $\gamma^\vee$ on $G^\vee$ is \textit{dual} to an OCRSF $\gamma$ on $G$ if no edge of $\gamma^\vee$ crosses an edge of $\gamma$. It is easy to see that $\gamma^\vee$ has the same number of cycles as $\gamma$ and each cycle has homology class $\pm [\eta]$, where $\eta$ is any cycle in $\gamma$. An OCRSF $\gamma$ has $2^k$ duals where $k$ is the number of cycles in $\gamma$, one for each choice of orientation of the dual cycles.\\
Given a pair $(\gamma,\gamma^\vee)$ of dual OCRSFs, define its weight to be $wt(\gamma,\gamma^\vee):=wt(\gamma)$. To each pair we associate a homology class,
$$[(\gamma,\gamma')]:=\frac{1}{2}\sum_{ \text{Cycles }\eta \text{ in }\gamma \cup \gamma^\vee}[\eta] \in H_1(\T,\Z).$$

\subsection{Newton polygon of the resistor network}
The \textit{Newton polygon} of the resistor network is 
$$
N=\text{Conv }\{[(\gamma,\gamma^\vee)] \in H_1(\T,\Z): (\gamma,\gamma^\vee) \text{ is a pair of OCRSFs}\}.
$$
$(\gamma,\gamma') \mapsto [(\gamma,\gamma')]$ associates to each pair of dual OCRSFs an integer lattice point in the Newton polygon. $N$ is always centrally symmetric and therefore we can center it at the origin. \\

Since $\mathcal L_G^{\text{flat}} \cong H^1(\T,\C^*)$, we have the natural pairing between homology and cohomology,
\begin{align*}
(\cdot,\cdot):H_1(\T,\Z) \times \mathcal L_G^{\text{flat}}&\ra \C^*
\end{align*}
We can rephrase Theorem \ref{pfnlap} as
$$
\text{det }\Delta=\sum_{(\gamma,\gamma^\vee)}wt(\gamma)([(\gamma,\gamma')],i) .
$$
$P(i):=\text{det }\Delta$ is called the \textit{characteristic polynomial}. The Newton polygon of the characteristic polynomial is
$$
\text{Conv}\{h \in H_1(\T,\Z): \text{Coefficient of }(h,i) \text{ is non-zero in }P(i)\},
$$
and it coincides with the Newton polygon of the resistor network.\\
If we fix a basis for $H_1(\T,\Z)$, we get isomorphisms
$H_1(\T,\Z)\cong \Z^2$ and $\mathcal L_G^{\text{flat}} \cong (\C^*)^2$. If $i \mapsto (z,w) \in (\C^*)^2$, $P(i)=P(z,w)$ is a Laurent polynomial in $z,w$.

\subsection{Temperley's bijection on the torus}
Let $G$ be a resistor network and let $\Gamma_G$ be the associated bipartite graph. 
\begin{lemma}[\cite{GK12}]The Newton polygon $N$ of the resistor network $G$ coincides with the Newton polygon of the dimer model on $\Gamma_G$.
\end{lemma}

Given a pair of dual OCRSFs $F=(\gamma,\gamma^\vee)$ on $G$, we can construct a dimer cover $M_F$ on $\Gamma_G$ using the rule: The oriented edge $e=\langle u,v \rangle$ is in $F$ if and only if the edge $\langle u,e\rangle$ is in $M_F$.

\begin{theorem}[Temperley's bijection on torus \cite{KPW00}]\label{temperley}
Let $(G,c)$ be a resistor network on a torus. $F \mapsto M_F$ is a weight preserving bijection from pairs of dual OCRSFs on $G$ to dimer covers on $\Gamma_G$ such that $[F]=[M_F]$ in $N$.
\end{theorem}

\subsection{Zig-zag paths and minimality for resistor networks}
An \textit{oriented zig-zag path} on a resistor network $G$ is a path that alternately turns maximally right or left at each vertex. Zig-zag paths on $G$ come in pairs with opposite orientations. We denote the set of zig-zag paths on $G$ by $\mathcal Z_G$. We say that $G$ is \textit{minimal} if the lift of any zig-zag path to the universal cover does not intersect itself and if the lifts of two different zig-zag paths intersect at most once.  \\
If $G$ is a minimal resistor network, associated to each $\alpha \in \mathcal Z_G$ is its homology class $[\alpha] \in H_1(\T,\Z)$. They correspond to integral primitive vectors on the boundary of the Newton polygon $N$ in cyclic order. \\
There is a natural bijection between $\mathcal Z_G$ and $Z_{\Gamma_G}$ that preserves the homology class.

\section{External OCRSFs}
We say that a pair of dual OCRSFs $F$ is \textit{external} if $[F]$ is a boundary lattice point of $N$. It is \textit{extremal} if $[F]$ is a vertex of $N$. We note that if $F=[(\gamma,\gamma^\vee)]$ is external, then the orientations of $\gamma$  and $\gamma^\vee$ are uniquely determined by $[F]$ and $[\gamma]=[\gamma^\vee]=[F]$. Therefore we can equivalently define external and extremal $OCRSFs$ on $G$ instead of pairs of dual $OCRSFs$. \\

For a vertex $v \in G$, we define the \textit{local zig-zag fan} $\Sigma_v$ at $v$ to be the complete fan of strongly convex rational polyhedral cones in $H_1(\T,\R)$ whose rays are generated by the homology classes of zig-zag paths through $v$ that turn maximally right at $v$.\\
The fan $\Sigma$ whose rays are generated by the homology classes of all zig-zag paths on $G$ is called the \textit{global zig-zag fan} of $G$. We have the natural map of fans $i_v:\Sigma \ra \Sigma_v$ for each $v \in G$. If $\sigma$ is a 2-dimensional cone in $\Sigma$, $i_v(\sigma)$ is contained in a unique two dimensional cone in $\Sigma_v$, which we shall denote by $\sigma_v$. $\sigma_v$ determines a unique edge $e$ adjacent to $v$ that is oriented away from $v$. Let $\gamma_{{\sigma}_v}$ be the 1-chain that is $1$ on $e$, $-1$ on $-e$ and $0$ on all other edges. We define
$$
\gamma_\sigma:=\sum_{v\in V(G)}\gamma_{\sigma_v}.
$$
From Temperley's bijection (Theorem \ref{temperley}) applied to Theorem \ref{extremaldimer}, we obtain:
\begin{lemma}\label{crsfextremal}
$\gamma_V:=\gamma_\sigma$ is the unique extremal OCRSF on $G$ such that $[\gamma_V]$ is the vertex $V$ of $N$ that corresponds to $\sigma$.
\end{lemma}

To a zig-zag path $\alpha \in Z_G$ we associate a 1-chain $\omega_\alpha$ that is $1$ on edges $e$ in $\alpha$ that are oriented in the same direction as $\alpha$ and $0$ on edges not in $\alpha$. If $\gamma$ is external, $[\gamma]$ lies on an edge $E$ of $N$, which corresponds to a family of zig-zag paths $\{\alpha_k\}$. Let $E=\langle V_1,V_2\rangle$, where $V_1,V_2$ are vertices of $N$ such that $V_2$ is the vertex after $V_1$ when the boundary of $N$ is traversed counterclockwise.\\

Using Temperley's bijection (Theorem \ref{temperley}), Theorem \ref{externaldimer} and the bijection between zig-zag paths on $G$ and $\Gamma_G$, we obtain:
\begin{lemma}\label{crsfexternal}
Let $A$ be a subset of the family of zig-zag paths $\{\alpha_k\}$ corresponding to $E$. The external OCRSFs on $E$ are of the form 
$$
\gamma_A:=\gamma_{V_1}+\sum_{\alpha_k \in A}\omega_{\alpha_k}.
$$
In particular, $\gamma_{V_2}=\gamma_{V_1}+\sum_{k}\omega_{\alpha_k}$, and the number of OCRSFs corresponding to a boundary lattice point of $N$ is a binomial coefficient.
\end{lemma}

\begin{corollary}\label{cyc}
Every external OCRSF is a union of cycles.
\end{corollary}
\begin{proof}
Suppose $\gamma_\sigma$ is an external OCRSF and let $v$ be a vertex of $G$. By construction, there is a single outgoing edge from $v$. We show that there is also a single incoming edge. Consider the fan $-\Sigma_v$ whose rays are generated by homology classes of zig-zag paths that turn maximally left at $v$ and let $i_v':\Sigma \ra -\Sigma_v$ be the natural map. $i_v'(\sigma)$ is contained in a unique two dimensional cone $\sigma_v'$ which corresponds to a unique edge $e$ oriented towards $v$. Define the 1-chain $\gamma_{\sigma_v}'$ to be $1$ on $e$ and $0$ on all other edges and define the 1-chain 
$$
\gamma_\sigma':=\sum_{v\in V(G)}\gamma_{\sigma_v}'.
$$
Let $e=\langle u,v \rangle$ be an edge in $G$ and let $\alpha_1$ and $\alpha_2$ be the two zig-zag paths through $e$ that turn maximally left at $v$. Then $\alpha_1$ and $\alpha_2$ turn maximally right at $u$ and therefore using Theorem \ref{span}, we have $\sigma'_v=\sigma_u$ which implies $\gamma_{\sigma_v}'=\gamma_{\sigma_u}$. Summing over all vertices, we get $\gamma_{\sigma}'=\gamma_{\sigma}$. It is clear from the definition of $\gamma_{\sigma}'$ that every vertex has a unique incoming edge. It follows that $\gamma_\sigma$ is a union of cycles.\\
By Lemma \ref{crsfexternal}, every external OCRSF is obtained from an extremal OCRSF $\gamma_V$ by adding cycles corresponding to some zig-zag paths and therefore is also a union of cycles. 

\end{proof}

\section{Spectral data}
A convex integral polygon $N$ determines a toric surface $\mathcal N$ along with an ample line bundle $\mathcal L$ on it. The global sections of $\mathcal L$ can be canonically identified with Laurent polynomials with Newton polygon $N$. Let $|\mathcal{L}|$ be the linear system of curves on $\mathcal N$ given by the vanishing loci of global sections of $\mathcal L$. Let $g=\text{number of interior lattice points in }N-1$. \\
Let $\mathcal{S}_N$ be the moduli space of triples $(C,S,\nu)$ such that $C$ is a curve in $|\mathcal{L}|$, $S$ is a degree $g$ effective divisor on $C$ and $\nu$ is a parameterization of the points at infinity of $C$. Let $G$ be a minimal resistor network associated to $N$ and $v$ a vertex of $G$. We have a natural rational map 
$$\rho_{G,v}: \mathcal R_N^0 \ra \mathcal S_N,$$
described on the affine chart $\mathcal R_G$ as follows:\\
$C_0$ is the spectral curve $\{(z,w) \in (\mathbb{C}^*)^2:\text{det }\Delta(z,w)=0\}$. Let $i:C_0 \hookrightarrow (\mathbb{C}^*)^2$ denote the inclusion. The Laplacian sits in the following exact sequence on $(\mathbb{C}^*)^2$:
\begin{equation}\label{es1}
\bigoplus_{v \in V} \mathcal{O}_{(\C^*)^2} \xrightarrow[]{\Delta} \bigoplus_{v \in V}\mathcal{O}_{(\C^*)^2} \ra \text{Coker }\Delta \ra 0.
\end{equation}
\begin{lemma}\label{linebundle}
$i^*\text{Coker }\Delta$ is a line bundle on $C_0$. 
\end{lemma}
\begin{proof}
$i^*\text{Coker }\Delta$ has one dimensional fibers over the non-singular points of  $C_0$(see \cite{CT79} Theorem 2.2). The fiber of $i^*\text{Coker }\Delta$ at $(1,1)$ is the space of harmonic functions on $G$ which is one dimensional because the only harmonic functions are constant. Since $C_0$ is reduced and $i^*\text{Coker }\Delta$ is a coherent sheaf of constant fiber dimension, it is locally free.
\end{proof}

The image of the section $\delta_{v}$ gives a section of $i^*\text{Coker }\Delta$. $S$ is the divisor of zeroes of this section.
$\nu$ is the parameterization of the points at infinity by zig-zag paths on $G$ such that the coordinate of the point at infinity associated to a zig-zag path is given by the monodromy around that zig-zag path. \\ 
Let $W \subset |\mathcal L|$ be the linear system of curves defined by sections $P(z,w)$ of $\mathcal L$ satisfying the following:
\begin{itemize}
\item $P(1,1)=0$ and the point $(1,1)$ is a node.
\item $\sigma:(z,w) \mapsto (\frac{1}{z},\frac{1}{w})$ is an involution on $\{P(z,w)=0\}$.
\end{itemize}

Let $\mathcal{S}_N'$ be the moduli space of triples $(C,S,\nu)$ such that $C$ is a curve in $W$, $S$ is a degree $g$ effective divisor on $C \setminus (1,1)$ satisfying 
\begin{equation}\label{divcondition}
S+\sigma(S)-q_1-q_2 \equiv K_{\hat{C}},
\end{equation}
where $\hat{C}$ is the normalization of $C$ and $\nu$ is a parameterization of the points at infinity.

\begin{theorem}
$\rho_{G,v}(\mathcal{R}^0_N) \subseteq \mathcal S_N'$.
\end{theorem}
The rest of this section is devoted to the proof of this theorem. 
Consider the following commuting diagram:
\begin{center}
\begin{tikzcd}
\hat{C_0} \arrow["\phi"]{r}{} \arrow[hookrightarrow]{d}[swap]{}
 &C_0 \arrow[hookrightarrow,"i"]{r}{} \arrow[hookrightarrow]{d}[swap]{}
 &(\mathbb{C}^*)^2 \arrow[hookrightarrow]{d}{} \\
 \hat{C} \arrow["\pi"]{r}{}
 &C\arrow[hookrightarrow,swap]{r}{}
 &\mathcal N
\end{tikzcd},
\end{center}
where $\phi$ and $\pi$ are the normalization maps. We pull back (\ref{es1}) using $\phi^* i^*$ and use right-exactness of pullback to get the following exact sequence on $\hat{C_0}$:
\begin{equation}\label{es2}
\bigoplus_{v \in V} \mathcal{O}_{\hat{C_0}} \xrightarrow[]{\phi^* i^*\Delta} \bigoplus_{v \in V}\mathcal{O}_{\hat{C_0}} \ra \text{Coker }\phi^* i^*\Delta \ra 0.
\end{equation}

\begin{theorem}[\cite{K17}]\label{harnackthm}
For the space $\mathcal{R}^0_N(\mathbb{R}_{>0})$ of positive real valued points of $\mathcal{R}^0_N$, $(C_0,S,\nu)\in \mathcal S_N'$. Moreover $C_0$ is a simple Harnack curve.
\end{theorem}

$P(z,w)=P(\frac{1}{z},\frac{1}{w})$ follows from $\Delta(z,w)=\Delta(\frac{1}{z},\frac{1}{w})^T$. $P(1,1)=0$ follows from the observation that the constant functions are discrete harmonic that is they are in the kernel of $\Delta(1,1)$. 
Differentiating the expression $P(z,w)=\sum_{\text{CRSFs }\gamma}wt(\gamma)(2-z^iw^j-z^{-i}w^{-j})$(where $(i,j)=[\gamma]$), we see that 

$$
\frac{\partial{P}(1,1)}{\partial z}=\frac{\partial{P}(1,1)}{\partial w}=0,
$$

so $(1,1)$ is a singular point. For all positive real points, Theorem \ref{harnackthm} tells us that $(1,1)$ is a node. Since nodes are characterized by non-vanishing of the Hessian, an open condition, $(1,1)$ is a node for all points in an open subset of $\mathcal{R}^0_N$.\\
\begin{lemma}
Let $\hat{C}$ be the normalization of $C$ and let $q_1,q_2 \in \hat{C}$ be the two points in the fiber over the node $(1,1)$. The divisor $S$ satisfies 

$$
S+\sigma(S)-q_1-q_2 \equiv K_{\hat{C}}.
$$
\end{lemma}

\begin{proof}
Let $Q(z,w)$ be the minor of $\Delta(z,w)$ with the row and column corresponding to $v_0$ removed. Consider the meromorphic 1-form
$$\omega=\frac{Q(z,w)dz}{zw\frac{\partial P(z,w)}{\partial w}}.$$
For smooth $(z,w) \in C$, we have $\text{corank }\Delta(z,w)=1$. Therefore we can write $\text{adj }\Delta(z,w)=U(z,w)V(z,w)^T$ for some $U(z,w) \in \text{Ker }\Delta(z,w),V(z,w) \in \text{Coker }\Delta(z,w)$. By definition, $S$ is the set of points in $C_0$ where the component $V(z,w) \cdot \delta_{v_0}$ of $V(z,w)$ vanishes. We have $\text{Ker }\Delta(z,w)\cong \text{Coker }\Delta(z,w)^T=\text{Coker }\Delta(\frac{1}{z},\frac{1}{w})$, so $\sigma(S)$ are the points where the component $U(z,w) \cdot \delta_{v_0}$ vanishes. Since $Q(z,w)= (U(z,w) \cdot \delta_{v_0})( V(z,w) \cdot \delta_{v_0})$, we have
$$
\text{div}_{C _0} Q(z,w)=S+\sigma(S),
$$

Since $C$ has a node at $(1,1)$, $\frac{\partial P(z,w)}{\partial w}$ has a simple zero at $(1,1)$ and so $\omega$ has simple poles at $q_1,q_2$. Therefore, the divisor of $\omega$ on the complement of the points at infinity is $S+\sigma(S)-q_1-q_2$, which has degree $2g-2$. It remains to identify the zeros and poles of $\omega$ at the points at infinity.\\
The order of vanishing of the 1-form
$$
\omega_{ij}:=\frac{z^{i-1}w^{j-1}dz}{\frac{\partial P(z,w)}{\partial w}}
$$
at the point at infinity corresponding to the primitive integral edge $E$ is given by the twice the signed area of the triangle formed by $E$ and the point $(i,j)$ minus one (where area is positive for points $(i,j)$ inside $N$). $Q(z,w)$ is the partition function of CRSFs on the graph $G'$ obtained from $G$ by deleting the vertex $v_0$. By Corollary \ref{cyc}, the Newton polygon of $Q(z,w)$ is strictly contained in $N$. Therefore the order of vanishing of $\omega$ must be non-negative at all points at infinity, that is $\omega$ has no poles at these points. The divisor of $\omega$ on the complement of the points at infinity has degree $2g-2$, which is the degree of the canonical class. Therefore $\omega$ must have an equal number of zeroes and poles at the points at infinity and therefore $\omega$ has no zeroes at infinity either.
\end{proof}

\subsection{Discrete Abel-Prym map}

Let $\mathcal Z=\{\alpha_1,...,\alpha_{2n}\}$ be an enumeration of oriented zig-zag paths in $G$ such that $\nu(\alpha_i)$ correspond to the primitive integral edges of the Newton polygon in cyclic order. We have $\sigma(\alpha_i)=\alpha_{n+i}$. 
Define $d':V(\tilde{G})\cup F(\tilde{G})  \ra \mathbb Z^{\mathcal Z}
$ as follows:\\
Set $d'(v)=0$ for some vertex $v$. For any vertex or face $u$, let $\tilde{\gamma}$ be a path from $v$ to $u$ in $\tilde{G}$ and let $\gamma$ be its image under the universal covering map $\tilde{G} \ra G$. Let  
$$
d'(u)=d'(v)+\sum_{\alpha \in \mathcal Z} ([\gamma],[\alpha])_\mathbb T \alpha, 
$$
where $(\cdot,\cdot)_\mathbb T$ is the intersection pairing on $H_1(\mathbb T,\Z)$. \\
Define the inclusion 
\begin{align*}
H_1(\T,\Z) &\hookrightarrow \Z^{\mathcal Z}\\
h &\mapsto \sum_{\alpha \in \mathcal Z}(h,[\alpha])_\T\alpha.
\end{align*}
Abusing notation, we will denote the homology class $h$ and its image in $\Z^{\mathcal Z}$ by the same letter $h$. Observe that $d'$ is equivariant with respect to the $H_1(\T,\Z)$ action, that is, 
$$
d'(h \cdot u)=h \cdot d'(u),
$$
for all $u \in V(\tilde{G})\cup F(\tilde{G})$. Define the discrete Abel map(\cite{Fock15}) $d:V(\tilde{G})\cup F(\tilde{G})  \ra Cl(\hat{C})$ as the composition $ \nu \circ d'$. For a homology class $h=(i,j)$ we have $\text{div}_{\hat{C}}z^iw^j=\nu(h)$, so $d$ descends to a well defined map $d:V(\tilde{G})\cup F(\tilde{G})  \ra Cl(\hat{C})$. We also the define the \textit{discrete Abel-Prym map }
 
\begin{align*}
d_P:  V(\tilde{G})\cup F(\tilde{G}) &\ra \text{Pr}(\hat{C},\sigma) \\
d_P&=\frac{1}{2} I_P \circ  d.
\end{align*}

The discrete Abel map provides us the following consistent way to extend (\ref{es2}):
\begin{equation}\label{es3}
\bigoplus_{v \in V} \mathcal{O}_{ \hat C}(d(v)-\sum_{\alpha_i \in \mathcal Z:v \in \alpha}\alpha_i -d(v_0)) \xrightarrow[]{\phi^*i^*\Delta} \bigoplus_{v \in V}\mathcal{O}_{\hat C}(d(v)-d(v_0)) \ra \text{Coker }\phi^*i^*\Delta \ra 0.
\end{equation}
We wish to identify $\text{Coker }\phi^*i^*\Delta$. 
\begin{lemma}
The divisor of the image of the section $\delta_{v_0}$ in $\text{Coker }\phi^*i^*\Delta$ in (\ref{es3}) restricted to $C_0$ is $S$.
\end{lemma}
\begin{proof}
The pullback $\phi^* i^*$ preserves zeroes and poles of sections.
\end{proof}
So we only have to identify the zeroes and poles of the image of $\delta_{v_0}$ at the points at infinity.
\begin{lemma}
$\delta_{v_0}$ has no zeroes or poles at infinity. 
\end{lemma}
\begin{proof}
Let $\alpha$ be an oriented zig-zag path. Let $x$ be a local parameter in a neighborhood $U$ of $\alpha$ disjoint from the other points at infinity with a simple zero at $\alpha$. We trivialize the line bundles in (\ref{es3}) as follows:
\begin{align*}
\mathcal O(-k \alpha )(U) & \xrightarrow[]{\cong}\mathcal O(U)\\
f & \mapsto x^{-k} f
\end{align*}
Let $z=a x^m + O(x^{m+1})$ and $w=b x^n+O(x^{n+1})$ be the expansions in the local coordinate $x$. Let us order the vertices so that the vertices on the zig-zag path appear first. Then the Laplacian matrix at $\alpha$ has the following block form:
$$
\Delta=\begin{pmatrix} 
\Delta_1 & B \\
0 & \Delta_2 
\end{pmatrix}+O(x),
$$
where $\Delta_1$ is the restriction of the Laplacian to the zig-zag path $\alpha$ and $\Delta_2$ is the restriction to the rest of the graph, and where $z$ and $w$ are replaced with $a$ and $b$ respectively. Since we are at $\alpha$, $\Delta_1$ is singular. Generically $\text{dim Ker }\Delta_1=1$ and $\Delta_2$ is invertible. In particular,  the fiber of $\text{Coker }\phi^*i^*\Delta$ at $\alpha$ is one dimensional. Combined with Lemma \ref{linebundle}, we get that $\text{Coker }\phi^*i^*\Delta$ is a line bundle. \\

Let $v \in \text{Ker }\Delta_1^*$. Then we have
$$
\text{Ker }\Delta^*=(v,-(\Delta_2^*)^{-1}B^*v)+O(x).
$$
Since generically none of the entries in $\text{Ker }\Delta^*$ is $0$, and since these entries are the cofactors of $\Delta$, we see that $\delta_{v_0}$ has no poles or zeros at $\alpha$. Since $\alpha$ was arbitrary, $\delta_{v_0}$ has no zeroes or poles at infinity.
\end{proof}
\begin{corollary}
$\text{Coker }\phi^*i^*\Delta=\mathcal{O}(S)$. For any other vertex $v$, let $S_v$ denote the divisor of the image of the rational section $\delta_v$ restricted to $\hat{ C_0}$. Then we have 
$$
\text{div}_{\hat{C}} \delta_v =S_v+ d(v)-d(v_0) \equiv S.
$$
\end{corollary}
Let $e=\frac{1}{2}\pi_1(I(S)-I(q_1)-I(q_2)-\pi^*\Delta_C)+d_P(v_0)$. Define for each vertex $v\in G$,
$$
\psi_v(x):=\frac{\eta(x+d_P(v)-e)}{\eta(d_P(v)-e)}E_{d(v)-d(v_0)}(x).
$$
By Theorem \ref{prt}, $\psi_v$ is a rational section of $\mathcal{O}(S)$ with divisor $S_v+d(v)-d(v_0)$.
\begin{lemma}\label{lemcoker}
The cokernel map is given by $\delta_v \mapsto  \psi_v$.
\end{lemma}
\begin{proof}
If $D$ is a generic degree $g$ effective divisor, the Riemann-Roch theorem tells us that $H^0(\hat{C},\mathcal O(D))$ is 1-dimensional. The cokernel map in (\ref{es3}) is given by  a collection of global sections of $\text{Hom}(\mathcal O(d(v)-d(v_0)),\mathcal O(S)) \cong \mathcal O(S+d(v_0)-d(v))\cong \mathcal O(S_v)$, and therefore uniquely determined up to scaling each component once we specify the image of $\delta_v$ for all $v$. The scaling is fixed by the requirement that the cokernel at $q_0$ and $q_1$ should be $(1,1,...,1)$. 
\end{proof}

\section{Inverse spectral map}
We now describe the normalization map $\pi$ explicitly.
\begin{lemma}
The following diagram commutes:
\begin{center}
\begin{tikzcd}
\hat{C}
\arrow[dashed,bend left]{drr}{x \mapsto (E_{(1,0)}(x),E_{(0,1)}(x))}
\arrow[bend right,swap]{ddr}{\pi}
\arrow[dashed]{dr}{} & & \\
& C_0 \arrow[hookrightarrow]{r}{} \arrow[hookrightarrow]{d}[swap]{}
& (\mathbb{C}^*)^2 \arrow[hookrightarrow]{d}{} \\
& C \arrow[hookrightarrow,swap]{r}{}
& \mathcal N
\end{tikzcd}
\end{center}
\end{lemma}
\begin{proof}
The functions $z$ and $w$ on $(\mathbb{C}^*)^2$ restrict to rational functions on $C$, which pull back to rational functions $\pi^*z$ and $\pi^*w$ on $\hat{C}$. We have 
$$
\text{div}_{\hat{C}}\pi^*z=\text{div}_{\hat{C}}E_{(1,0)}(x),
$$
so they agree up to multiplication by a constant. Since $E_{(1,0)}(q_1)=\pi^*z(q_1)=1$, the constant is $1$, and therefore we have $ \pi^*z=E_{(1,0)}(x)$. By the same argument applied to $w$, we get $ \pi^*w=E_{(0,1)}(x)$.
\end{proof}
\begin{figure}\label{condfig}
\centering
\includegraphics[width=0.3\textwidth]{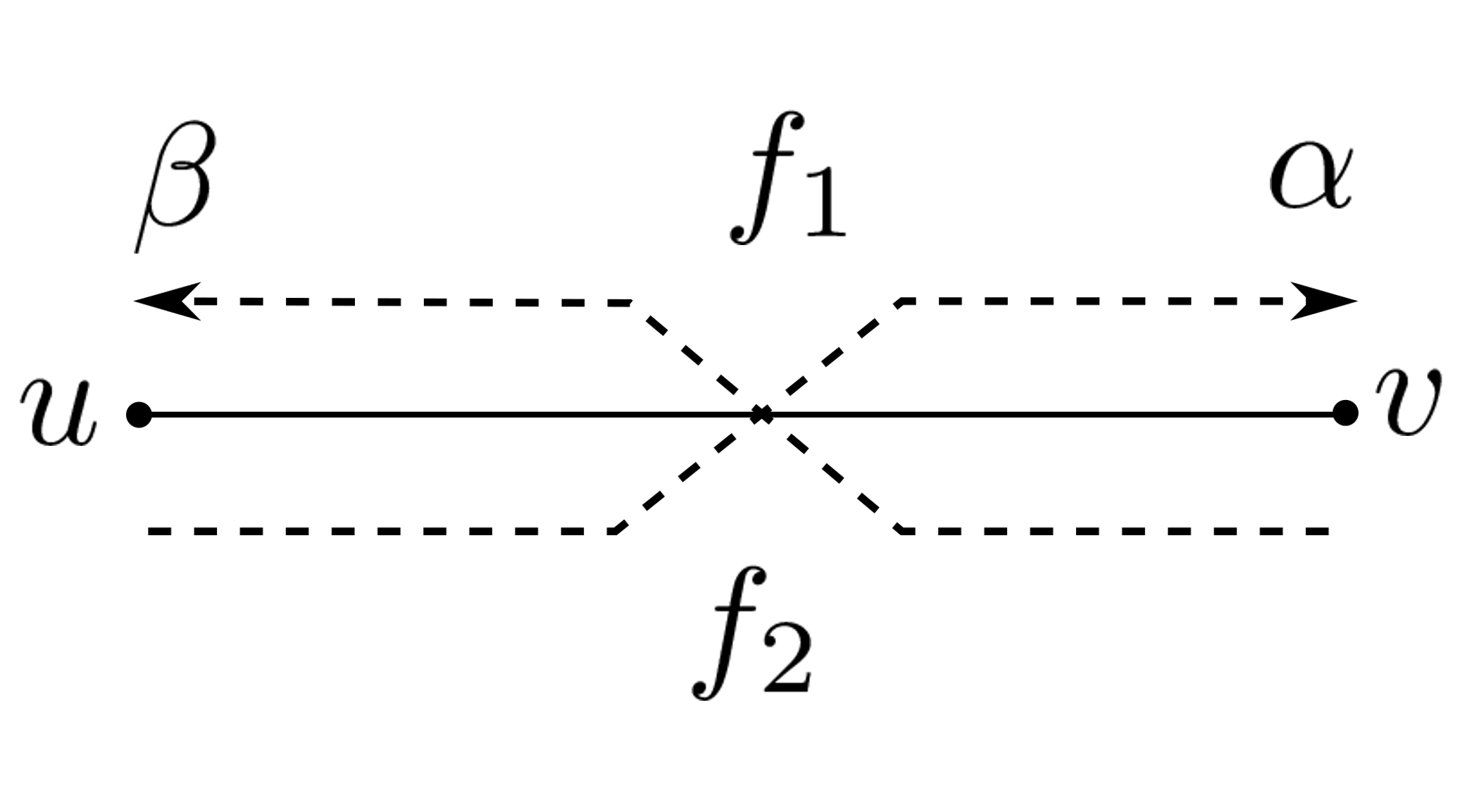}
\caption{Vertices, faces and zig-zag paths in the definition of the conductance function.}\label{condfig}
\end{figure}
Let $uv$ be an edge in $\tilde{G}$, $f_1$ and $f_2$ be the faces adjacent to $uv$ and let $\alpha,\beta$ be the zig-zag paths as shown in Figure \ref{condfig}. Define the conductance function
\begin{equation}\label{invmap}
c_{u,v}:=\frac{\eta(d_P(u)-e)\eta(d_P(v)-e
)}{\eta(d_P(f_1)-e)\eta(d_P(f_2)-e)}\frac{E(\alpha,\beta)}{E(\alpha,\beta')}.
\end{equation}
\begin{lemma}
$c_{u,v}$ has the following properties:
\begin{enumerate}
\item $c_{u,v}=c_{v,u};$
\item $c_{u,v}$ is compatible with taking the dual graph, that is, $c_{f_1,f_2}=1/c_{u,v}$. 
\item $c_{u,v}$ is $H_1(\T,\Z)$-periodic and therefore descends to a conductance function $c$ on $G$.

\end{enumerate}
\end{lemma}
\begin{proof}
\begin{enumerate}
\item Follows from the symmetry $E(\alpha,\beta)=E(\alpha',\beta')$.
\item Clear.
\item Let $h \in H_1(\T,\Z)$. We have 
\begin{align*}
I_P(d_P(u+h)-d_P(u))&=\frac{1}{2}\pi_1 I(h)=0,
\end{align*}
since $h=(i,j)=\text{div}_{\hat{C}}z^iw^j$. 
\end{enumerate}
\end{proof}

\begin{theorem}

The rational map $\rho_{G,v_0}:(C,S,\nu)\mapsto V(c)$ is the inverse of $\kappa_{G,v_0}$. Therefore $\mathcal{R}^0_N $ is birational to $\mathcal S_N'$.

\end{theorem}

\begin{proof}

\begin{figure}\label{condfig2}
\centering
\includegraphics[width=0.5\textwidth]{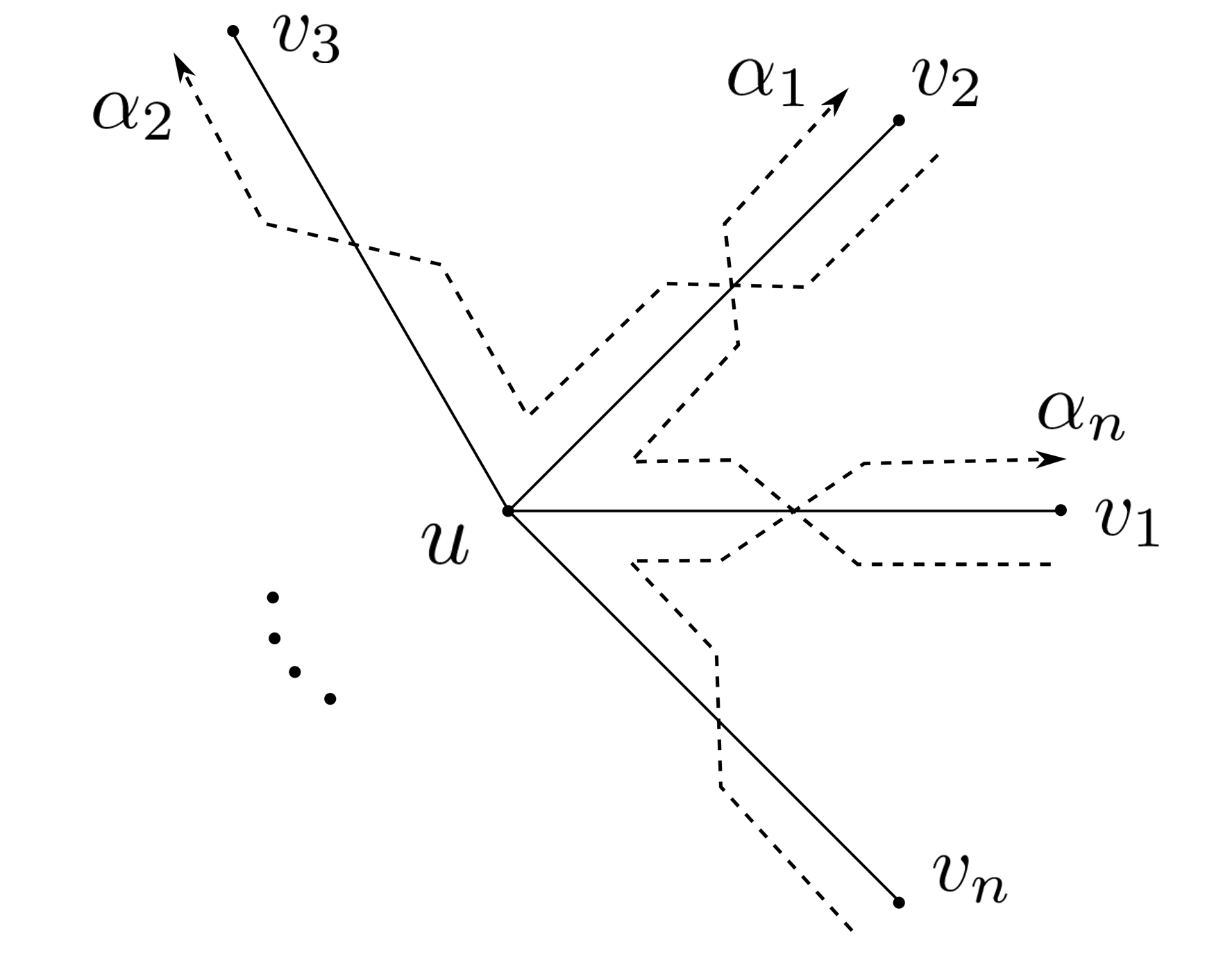}
\caption{Local configuration near a vertex $u$.}\label{condfig2}
\end{figure}
\begin{enumerate}
\item $\kappa_{G,v_0} \circ  \rho_{G,v_0}=\text{id}:$\\Let $u$ be a vertex in $\Gamma$ and let $v_1,...,v_n$ be the vertices adjacent to $u$ in $G$. Let $\alpha_1,...,\alpha_n$ be the zig-zag paths as shown in Figure \ref{condfig2}. Note that
$$
i_{v,u}^{-1}\psi_v(x)=\psi_u(x).
$$

Using Theorem \ref{fqi} with $z=q_1,t=d_P(u)-e,x_k=\alpha_k$, we get 
\begin{equation}\label{condsum}
\sum_{v_k \sim u}c_{u,v_k}=\frac{\eta\left(d_P(u)-e-\sum_{i=1}^k \alpha_k\right)\eta(d_P(u)-e)^2}{\prod_{k=1}^n\eta(d_P(u)-e-\alpha_k)}\prod_{k=1}^n \frac{E(\alpha_k,\alpha_{k+1})}{E(\alpha_k,\alpha'_{k+1})}. 
\end{equation}
Using Theorem \ref{fqi} with $z=x,t=d_P(u)-e,x_k=\alpha_k$ and (\ref{condsum}), we get
$$
\sum_{v_k \sim u}c_{u,v_k} (\psi_{u}(x)-i_{v_k,u}^{-1}\psi_{v_k}(x))=0,
$$
so the following is sequence is exact:

\begin{equation*}
0 \rightarrow \text{Ker }\phi^*i^* \Delta^T \xrightarrow[]{1 \mapsto (\psi_v)_{v}}\bigoplus_{v \in V}\mathcal{O}_{\hat C}(-d(v)+d(v_0)) \xrightarrow[]{\phi^*i^*\Delta^T} \bigoplus_{v \in V} \mathcal{O}_{\hat C}(-d(v)+\sum_{\alpha_i \in \mathcal Z:v \in \alpha}\alpha_i +d(v_0)).
\end{equation*}
Since this is the transpose of (\ref{es3}), the cokernel map in (\ref{es3}) is $\delta_v \mapsto \psi_v$ and we recover $S=\text{div}_{\hat{C_0}}\psi_{v_0}$ as the divisor.
\item $\rho_{G,v_0} \circ \kappa_{G,v_0}=\text{id}$:\\
Suppose $c'$ is a conductance function such that $\kappa_{G,v_0}(c')=(C,S,\nu)$. By Lemma \ref{lemcoker}, the cokernel map is determined by $S$ and is given by $\delta_v \mapsto \psi_v$. Taking transpose, the equation of $\phi^* i^* \Delta^T$ becomes 
$$
\sum_{v_k \sim u}c'_{u,v_k} (\psi_{u}(x)-i_{v_k,u}^{-1}\psi_{v_k}(x))=0.
$$
Since the coefficients of the quadrisecant identity are uniquely determined up to a constant, comparing with Theorem \ref{fqi} with $z=x,t=d_P(u)-e,x_k=\alpha_k$, we see that $c'$ agrees with $c$ up to a multiplicative constant.
\end{enumerate}
 \end{proof}

\section{Compatibility with $Y-\Delta$ transformations}

\begin{figure}\label{et}
\centering

\includegraphics[width=0.7\textwidth]{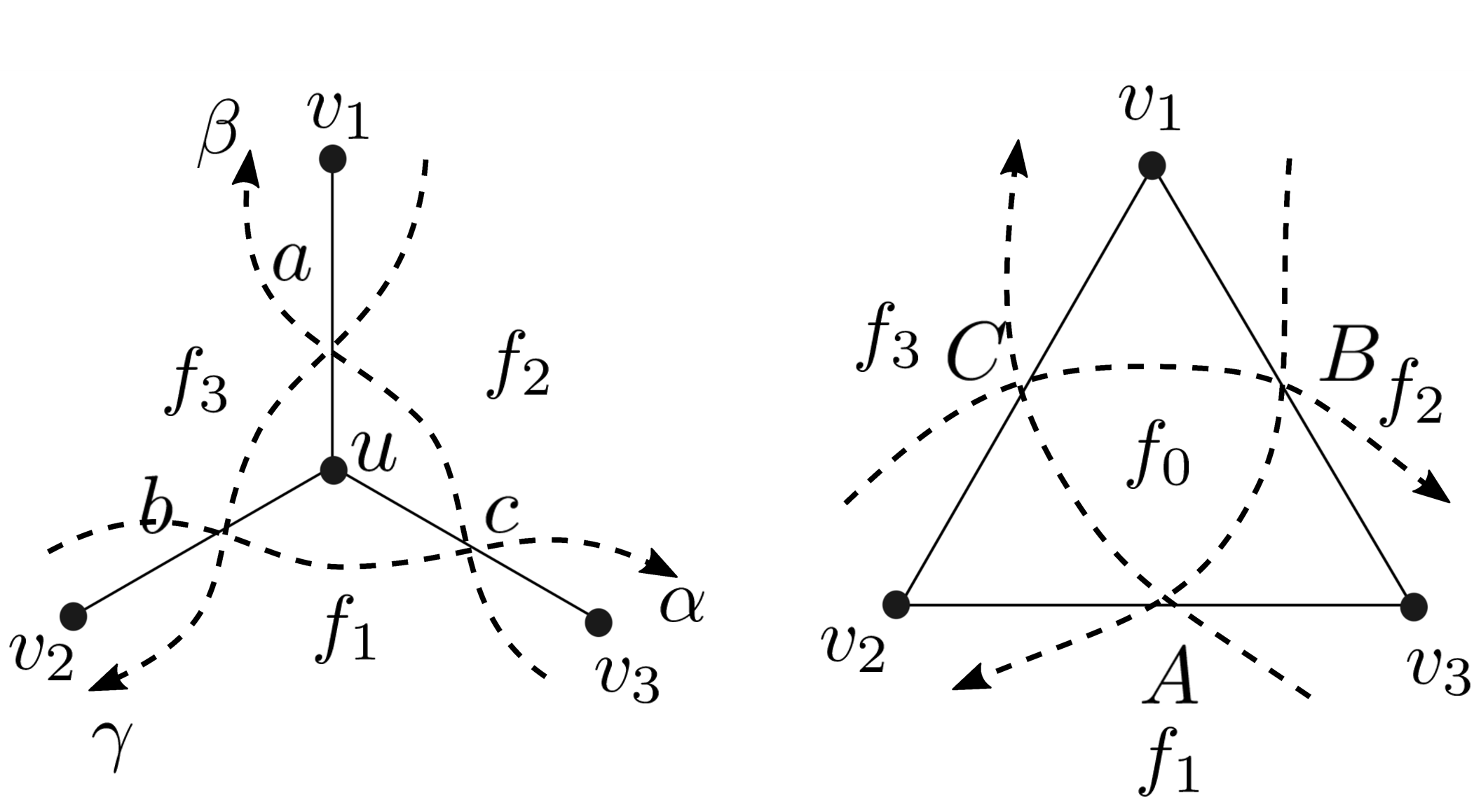}

\caption{Y-Delta transformation.}\label{et}
\end{figure}

A $Y-\Delta$ transformation is induced by sliding a zig-zag path through the crossing of two other zig-zag paths as shown in Figure \ref{et}. Therefore discrete Abel and discrete Abel-Prym maps $d,d_P$ on $G_1$ induce discrete Abel and discrete Abel-Prym maps on $G_2$, which we will also denote by $d,d_P$ respectively.

\begin{theorem}\label{ydcomp}

Let $G_1 \ra G_2$ be a $Y-\Delta$ transformation and let $v_1$ and $v_2$ be vertices of $G_1$ and $G_2$ respectively. The following diagram commutes:\\
\begin{center}
\begin{tikzcd}
& \mathcal{R}^0_{N}\arrow{dr}{\kappa_{G_2,v_2}} \arrow[dl,"\kappa
_{G_1,v_1}"'] & \\
\mathcal{S}'_N \arrow[rr,"s"] & & \mathcal{S}'_N
\end{tikzcd}
\end{center}
The birational map $s$ is defined as $(C,S_1,\nu_1) \mapsto (C,S_2,\nu_2)$, where
\begin{enumerate}
\item There is a natural bijection between zig-zag paths on $G_2$ and $G_1$ induced by $Y-\Delta$ transformation. $\nu_2$ is obtained by composing this bijection with $\nu_1$.
\item $S_2$ is the generically unique degree $g$ effective divisor satisfying $S_2\equiv S_1 + d(v_1)-d(v_2).$ 
\end{enumerate}
\end{theorem}
\begin{proof}
The $Y-\Delta$ transformation preserves the spectral curve. The local picture is shown in Figure \ref{et}. Let $e=\frac{1}{2}\pi_1(I(S_1)-I(q_1)-I(q_2)-\pi^*\Delta_C)+d_P(v_1)$. We show that $\kappa
_{G_1,v_1}^{-1}= \kappa
_{G_2,v_2}^{-1} \circ s$. We have 
\begin{align*}
a&=\kappa_{G_1,v_1}^{-1}(C,S_1,\nu_1)_{u v_1}=\frac{\eta(d_P(u)-e)\eta(d_P(v_1)-e
)}{\eta(d_P(f_2)-e)\eta(d_P(f_3)-e)}\frac{E(\beta,\gamma)}{E(\beta,\gamma')};\\
b&=\kappa_{G_1,v_1}^{-1}(C,S_1,\nu_1)_{u v_2}=\frac{\eta(d_P(u)-e)\eta(d_P(v_2)-e
)}{\eta(d_P(f_1)-e)\eta(d_P(f_3)-e)}\frac{E(\gamma,\alpha)}{E(\gamma,\alpha')};\\
c&=\kappa_{G_1,v_1}^{-1}(C,S_1,\nu_1)_{u v_3}=\frac{\eta(d_P(u)-e)\eta(d_P(v_3)-e
)}{\eta(d_P(f_1)-e)\eta(d_P(f_2)-e)}\frac{E(\alpha,\beta)}{E(\alpha,\beta')}.
\end{align*}
Note that by the definition of $s$, 
\begin{align*}
&\frac{1}{2}\pi_1(I(S_2)-I(q_1)-I(q_2)-\pi^*\Delta_C)+d_P(v_2)\\
&=\frac{1}{2}\pi_1(I(S_1+d(v_1)-d(v_2))-I(q_1)-I(q_2)-\pi^*\Delta_C)+d_P(v_2)\\
&=e
\end{align*}
Therefore
$$
A=\kappa
_{G_2,v_2}^{-1} \circ s(C,S_2,\nu_2)_{v_2v_3}=\frac{\eta(d_P(v_2)-e)\eta(d_P(v_3)-e
)}{\eta(d_P(f_0)-e)\eta(d_P(f_1)-e)}\frac{E(\gamma,\alpha')}{E(\gamma,\alpha)}.
$$
Equation (\ref{condsum}) becomes $$
a+b+c=\frac{\eta(d_P(u)-e)^2\eta(d_P(f_0)-e
)}{\eta(d_P(f_1)-e)\eta(d_P(f_2)-e)\eta(d_P(f_3)-e)}\frac{E(\alpha,\beta)E(\beta,\gamma)E(\gamma,\alpha)}{E(\alpha,\beta')E(\beta,\gamma')E(\gamma,\alpha')}.
$$
Plugging in these expressions, we see that $\frac{bc}{a+b+c}=A$, which is the transition map between the $G_1$ and $G_2$ affine charts.
\end{proof}

\section{Discrete integrable systems from $Y-\Delta$ moves}

Let $T$ be a sequence of $Y-\Delta$ moves on a graph $G$ such that the resulting graph $T \cdot G$ is isomorphic to $G$ as graphs. Let $\phi_T:G \ra T \cdot G$ be the isomorphism. The composition 
\begin{align*}
\mathcal R^0_N \supset \mathcal R^0_G \ra \mathcal R^0_{T\cdot G} \xrightarrow[]{\simeq} \mathcal R^0_G \subset \mathcal R^0_N    
\end{align*}
defines a birational automorphism of $\mathcal R^0_N$, which we denote by $\mu_T$. It is a cluster modular transformation as defined in \cite{FG03b}. Using Theorem \ref{ydcomp}, we construct the follwing commuting diagram:
\begin{center}
\begin{tikzcd}
\mathcal R^0_N \supset \mathcal{R}^0_G  \arrow[r] \arrow[d,"\kappa_{G,v}"] \arrow[bend left=30]{rr}{\mu_T}
& \mathcal{R}^0_{T \cdot G} \arrow[d,"\kappa_{T \cdot G,\phi_T^{-1}(v)}"] \arrow[r,"\simeq"] & \mathcal{R}^0_{G}\subset \mathcal R^0_N  \arrow[d,"\kappa_{G,v}"]\\
 \mathcal{S}'_N \arrow[r,"s"]
& \mathcal{S}'_N \arrow[r,"t"] & \mathcal{S}'_N
\end{tikzcd},
\end{center}
where $s$ is the map in Theorem \ref{ydcomp} and $t$ is the natural map induced by the graph isomorphism $\phi_T$, that is  $(C,S,\nu) \mapsto (C,S,\nu')$, where $\nu'$ is obtained from $\nu$ by composing with $\phi_T$. We have shown:
\begin{theorem}\label{lin}
The following diagram commutes:
\begin{center}
\begin{tikzcd}
\mathcal{R}^0_N \arrow[r,"\mu_T"] \arrow[d,"\kappa_{G,v}"] 
& \mathcal{R}^0_N \arrow[d,"\kappa_{G,v}"] \\
\mathcal{S}'_N \arrow[r,"s_T"]
& \mathcal{S}'_N
\end{tikzcd},
\end{center}
where the birational map $s_T$ is defined as $(C,S,\nu) \mapsto (C,S_T,\nu_T)$ where $S_T$ is the (generically) unique degree $g$ effective divisor satisfying $S_T \equiv S +d(v)-d(\phi_T^{-1}(v))$ and $\nu_T=\nu \circ \phi_T^{-1}$. 
\end{theorem}

For a fixed $(C)$, the fiber of the projection $(C,S,\nu) \mapsto (C)$ over $(C)$ is a cover of the space of degree $g$ effective divisors on $C$ satisfying (\ref{divcondition}), which is birational to a cover of $\text{Prym}(\hat{C},\sigma)$. Therefore Theorem \ref{lin} tells us that the discrete integrable system arising from $T$ is linearized on a cover of  $\text{Prym}(\hat{C},\sigma)$.

\section{A conjecture}
Let $G$ be a minimal resistor network, $\Gamma_G$ be the associated bipartite graph. Recall the dimer spectral data $\kappa_{\Gamma_G,v}:\mathcal{X}_N^0 \ra \mathcal{S}_N$ as defined in \cite{GK12} Proposition 7.2. By \cite{GK12} Theorem 1.4, \cite{Fock15}, $\kappa_{\Gamma_G,v}$ is a birational isomorphism. We conjecture that the map $t$ that makes the diagram below commute is $(C,S,\nu) \mapsto (C,S+(1,1),\nu)$.
\begin{center}
\begin{tikzcd}
\mathcal{R}^0_N \arrow[r,"\kappa_{G,v}"] \arrow[d] 
& \mathcal{S}_N' \arrow[d,"t"] \\
\mathcal{X}_N^0 \arrow[r,"\kappa_{\Gamma_G,v}"]
& \mathcal{S}_N
\end{tikzcd}
\end{center}

\section{Appendix}

For background on the material collected here, see \cite{Fay73}, \cite{Fay89}, \cite{Tata1},\cite{Tata2}, \cite{Taim97}. Let $\pi: \hat{C} \ra C$ be a ramified double covering of genus $\hat{g}$ of a smooth curve of genus $g$ with branch points $q_1,q_2$. By the Riemann-Hurwitz theorem, $\hat{g}=2g$. Let $\sigma: \hat{C} \ra \hat{C}$ be the involution permuting the branches of the covering with fixed points at $q_1,q_2$ and let $x'=\sigma(x)$ denote the conjugate point of $x \in \hat{C}$. 
We can choose a canonical homology basis for $H_1(\hat{C},\Z)$ 
$$
A_1,B_1,A_2,B_2,...,A_{2g},B_{2g},
$$
such that 
$(\pi_*(A_i),\pi_*(B_i))_{i=1}^{g}$ is a basis for $H_1(C,\Z)$ and such that 
$$
\sigma(A_k)+A_k=\sigma(B_k)+B_k=0, \quad 1 \leq k \leq g.
$$
If the dual basis of holomorphic differentials on $\hat{C}$ is 
$$
u_1,...,u_{2g},
$$
then for $1 \leq k \leq g$ we have 
$$
\sigma^* u_{k}+u_{g+k}=0.
$$
A holomorphic differential $\omega$ on $\hat{C}$ is called a Prym differential if $\sigma^*(\omega)+\omega=0$. For $1 \leq k \leq g$,
$$
\omega_k=\sigma^* u_k + u_k
$$
is a basis for Prym differentials on $\hat{C}$. Let $\Pi$ be the matrix of periods of the Prym differentials around the $b-$cycles of $\hat{C}$: 
$$
\Pi_{jk}=\int_{b_k} u_j.
$$
The \textit{Prym variety} $\text{Pr}(\hat{C},\sigma)$ is defined to be
$$
\frac{\mathbb{C}^g}{\Z^g + \Pi \Z^g}.
$$
Let $E(x,y)$ denote the prime form on $\hat{C}$. $E(x,y)$ has the symmetry $E(x,y)=E(x',y')$ for all $x,y \in \hat{C}$. Let $Cl(\hat{C})$ denote the divisor class group of $\hat{C}$. For a divisor $D=\sum_i a_i-\sum_j b_j \in Cl(\hat{C})$, define $$E_D(x):=
\frac{\prod_i E(x,a_i)}{\prod_j E(x,b_j)}.$$ It is a section of the line bundle associated to $D$ with divisor $D$.

Let $\hat{J},J$ be the Jacobians of $\hat{C},C$ respectively and let  $I: \hat{C} \ra \hat{J}$ be the Abel map with base-point $p_0\in \hat{C}$. By Riemann's theorem, we have $\hat{J} =I(\text{Symm}^{2g}\hat{C})$ and the involution $\sigma$ induces an involution $
\sigma_*:\hat{J} \ra \hat{J}$: Given $\zeta \in \hat{J}$, let $D \in \text{Symm}^{2g}(\hat{C})$ such that $I(D)=\zeta$ and let
$\sigma_*(\zeta)=I(\sigma(D))$. In coordinates, $\sigma_*$ is given by 
$$
(z_1,...,z_{2g})\mapsto (-z_{g+1},...,-z_{2g},-z_1,...,-z_g).
$$

The Prym variety is embedded by $\phi:\text{Pr}(\hat{C},\sigma) \hookrightarrow \hat{J}:$  $$
(z_1,...,z_g) \mapsto (z_1,...,z_g,z_1,...,z_g).
$$.
We also have projections $\pi_1:\hat{J}\ra  \text{Pr}(\hat{C},\sigma)$ and $\pi_2:\hat{J} \ra J$ given by 

\begin{align*}
\pi_1(z_1,...,z_{2g})&=(z_1+z_{g+1},...,z_g+z_{2g})\\
\pi_2(z_1,...,z_{2g})&=(z_1-z_{g+1},...,z_g-z_{2g}).
\end{align*}
Define the Abel-Prym map with base-point $q_1$: 
\begin{align*}
I_{P}:\hat{C} &\ra \text{Pr}(\hat{C},\sigma)\\
x &\mapsto \left( \int_{q_1}^{x} \omega_1,...,\int_{q_1}^{x} \omega_g \right), \text{ for } x \in \hat{C}.
\end{align*}
Note that $I_P=\pi_1 \circ I$.
Let $\eta(z)$ be the theta function on $\text{Pr}(\hat{C},\sigma)$. Note that for $e \in \text{Pr}(\hat{C},\sigma)$, we have 
$$
e=\frac{1}{2}\pi_1(\phi(e)).
$$

\begin{theorem}\label{prt}
If $e \in \text{Pr}(\hat{C},\sigma)$, then either $\eta(I_P(x)-e) \equiv 0$ for all $x \in \hat{C}$ or $\text{div} _{\hat{C}}\eta(I_P(x)-e)=D$ is a degree $\hat{g}$ effective divisor satisfying
$$
\phi(e) = I(D)-I(q_1)-I(q_2)-\pi^* \Delta_C \quad \text{in }\hat{J},
$$
where $\Delta_C \in J$ is the vector of Riemann constants on $C$, and 
$$
D+\sigma(D)-q_1-q_2 \sim K_{\hat{C}},
$$
where $K_{\hat{C}}$ is the canonical class of $\hat{C}$. Moreover, $D$ is determined by these conditions.
\end{theorem}

\begin{theorem}[Fay's quadrisecant identity \cite{Fay89}]\label{fqi}
Let $t \in Pr(\hat{C},\sigma)$, $z\in \hat{C}$  and suppose $x_k \in \hat{C}$ for $k \in \Z/n \Z$.  
\begin{align*}
\sum_{k=1}^{n}\frac{\eta(t+I_P(z)-I_P(x_k)-I_P(x_{k+1}))}{\eta(t-I_P(x_k))\eta(t-I_P(x_{k+1}))}\frac{E(x_k,x_{k+1})}{E(x_k,x'_{k+1})}\frac{E(z,x'_k)E(z,x'_{k+1})}{E(z,x_k)E(z,x_{k+1})}\\
=\frac{\eta\left(t-\sum_{i=1}^k I_P(x_k)\right)\eta(t+I_P(z))}{\prod_{k=1}^n\eta(t-I_P(x_k))}\prod_{k=1}^n \frac{E(x_k,x_{k+1})}{E(x_k,x'_{k+1})}.
\end{align*}
\end{theorem}

\Addresses

\end{document}